\UseRawInputEncoding
\documentclass[11pt]{article}
\usepackage{float}
\usepackage{color}
\usepackage{subfigure}
\usepackage{amsmath}
\usepackage{amsthm}
\usepackage{amsfonts}
\usepackage{amssymb,latexsym}
\usepackage[all]{xy}
\usepackage{graphicx}
\usepackage[mathscr]{eucal}
\usepackage{verbatim}
\usepackage{hyperref}

\theoremstyle{plain}

    \newtheorem{thm}{Theorem}[section]
       
       \newtheorem{lem}{Lemma}[section]
       \newtheorem{rem}{Remark}[section]

\numberwithin{equation}{section}

\usepackage{graphicx}

\begin{document}
\title{Hausdorff  dimension  of random attractors for a stochastic delayed parabolic equation  in Banach spaces}

\author{Wenjie Hu$^{1,2}$,  Tom\'{a}s
Caraballo$^{3,4}$\footnote{Corresponding author.  E-mail address: caraball@us.es (Tom\'as Caraballo)}, Yueliang Duan$^{5}$.
\\
\small  $^1$The MOE-LCSM, School of Mathematics and Statistics,  Hunan Normal University,\\
\small Changsha, Hunan 410081, China\\
\small  $^2$Journal House, Hunan Normal University, Changsha, Hunan 410081, China\\
\small $^3$Dpto. Ecuaciones Diferenciales y An\'{a}lisis Num\'{e}rico, Facultad de Matem\'{a}ticas,\\
\small  Universidad de Sevilla, c/ Tarfia s/n, 41012-Sevilla, Spain\\
\small $^4$Department of Mathematics, Wenzhou University, \\
\small Wenzhou, Zhejiang Province 325035, China\\
\small  $^5$Department of Mathematics, Shantou University, Shantou 515063, China
}

\date {}
\maketitle

\begin{abstract}
The main purpose of this paper is to give an upper bound  of  Hausdorff dimension   of random attractors for a stochastic delayed parabolic equation  in Banach spaces. The estimation of dimensions of random attractors are obtained by combining the squeezing property and a covering lemma of finite subspace of Banach spaces,  which generalizes the method established in Hilbert spaces. Unlike the existing works, where orthogonal projectors with finite ranks applied for  proving the squeezing property of stochastic partial differential equations in Hilbert spaces, we adopt the state decomposition of phase space based on the exponential dichotomy of the the linear deterministic part of the studied SDPE to obtain similar squeezing property due to the lack of smooth inner product geometry structure. The obtained dimension of the random attractors depend only on the spectrum of the linear part and the random Lipschitz constant of the nonlinear term, while not relating to the compact embedding of the phase space to another Banach space as the existing works did.
\end{abstract}

\bigskip


\section{Introduction}
For infinite dimensional random dynamical systems (RDSs),  the existence of random attractors can  reduce the essential parts of the random flows to random compact sets. Furthermore, if the attractors have finite  Hausdorff dimension or fractal dimension, then the attractors can be described by a finite number of parameters and hence the limit dynamics of the infinite dimensional RDSs are likely to be studied by the concepts and methods of finite dimensional RDSs. The study of random attractors for RDSs dates back to the pioneer works \cite{9,10,FS}, where  Crauel,  Flandoli and Schmalfu\ss, amongst others, generalized the concept of global attractors of infinite dimensional dissipative systems and  established the basic framework of random attractors for infinite dimensional RDSs. Since then,  the existence, dimension estimation and qualitative properties of random attractors for various stochastic evolution equations have been investigated by many  researchers. For example, for the stochastic partial differential equations (SPDEs) without time delay, Caraballo et al. \cite{11}, Gao et al. \cite{13} and Li and Guo \cite{12} explored the existence of global attractors on bounded domains. In \cite{24}, \cite{15}  and \cite{14}, the authors obtained the existence of global attractors on unbounded domains. For SPDEs with  delay, the existence of random attractors and  their qualitative properties have been extensively and intensively studied in \cite{17,16,18,HZ20,HZT,LG20,25} and the references therein.

Criteria for the  finite Hausdorff dimensionality of attractors for deterministic fluid dynamics models have been derived by  Douady and  Oesterle \cite{DO}, which was later generalized by  Constantin,  Foias and  Temam \cite{CFT} and \textcolor{red}{\cite{ZZW}}. Then, it was further extended to the stochastic case in \cite{CF} and \cite{SB}, where the RDS is first linearized and  the global Lyapunov exponents of the  linearized mapping is then examined.  Debussche  showed that the random attractors of many RDSs  have finite Hausdorff dimension by an ergodicity argument in \cite{DA97} and further gave a precise bound on the dimension by combining the method of   linearization and Lyapunov exponents in \cite{DA98}. \cite{CCL,FX08,L,LR,C37} considered   the fractal dimensionality of random sets and random attractors. In the recent works \cite{XHM} and \cite{14}, the authors proved the finiteness of fractal dimension of random attractor for SPDEs with linear multiplicative white noise by extending the idea of \cite{DA97} to obtain  the existence of random exponential attractors.

Despite the fact that the  finite Hausdorff dimension  and fractal dimension  of attractors for  abstract RDSs and applications to SPDEs in Hilbert spaces have been extensively and intensively studied, to our best knowledge, the estimation of dimensions of  stochastic differential equations with delays, i.e., the stochastic partial functional differential equations (SPFDEs) have not been extensively studied since their natural phase space are Banach spaces. The lack of smooth inner product causes the existing methods can not be directly applied. Therefore, in our recent work \cite{HZC}, we extend the method established in \cite{SW91} to recast SPFDEs into an auxiliary Hilbert space and adopt the method established in \cite{DA98} to give upper bound of the Hausdorff  and fractal dimensions of a delayed reaction-diffusion equation by requiring the nonlinear term to be twice continuous and the first derivative satisfies certain conditions. Nevertheless, the natural phase space for SPFDEs are Banach spaces  and recast SPFDEs into Hilbert spaces are unnatural and the conditions are quite strong and hence one naturally wonders what can we say about the Hausdorff dimension of random attractors for SPFDEs in their natural phase space, i.e. the Banach spaces? Here, we extend the squeeze method in \cite{DA97}  for estimating dimensions of SPDEs in Hilbert spaces to SPFDEs in Banach  spaces by combing the squeezing property and the covering of finite subspace of Banach spaces. Unlike \cite{DA97}, where orthogonal projector in Hilbert space and variational technique are adopted to obtain squeezing property,   we prove similar squeezing property by semigroup approach and a phase space decomposition  based on the exponential dichotomy of linear part.

It should be pointed out that in \cite{C5} and \cite{C32},  the authors  proved the existence of random exponential attractors and uniform exponential attractors of random dynamical systems in Banach spaces  based on a smoothing property of the systems, which indicates the finite fractal dimensionality of  the systems.  However, the fractal dimension obtained  by smoothing property may depend on the choice of another embedding space, which may vary from space to space. Furthermore, the dimension estimation depends on the entropy number between two spaces for which is generally quite difficult to obtain an explicit bound.   Here, the dimension obtained by our method only depends on the inner characteristics of the studied equation while not depend the entropy number between two spaces.

We consider the following stochastic delayed parabolic equation with additive noise
\begin{equation}\label{1}
\left\{\begin{array}{l}\frac{\partial u}{\partial t}(x, t)= \Delta u(x, t)-\mu u(x, t)-\sigma u(x,t-\tau)+F\left(u(x,t)\right)+f \left(u(x,t-\tau)\right)\\
\quad \quad \quad \quad +\sum_{j=1}^{m} g_{j}(x)\frac{\mathrm{d} \omega_{j}(t)}{\mathrm{d}t}, t>0, x \in \mathcal{O},\\ u_0(x,s)=\phi(x, s),-\tau \leq s \leq 0, x \in \mathcal{O},\\
u_0(x,t)=0,-\tau \leq t, x \in \partial \mathcal{O},\end{array}\right.
\end{equation}
where $ \mathcal{O}\subseteq\mathbb{R}^N$ is a bounded open domain with smooth boundary $\partial \mathcal{O}$, $\{\omega_{j}\}_{j=1}^{m}$ are mutually independent two-sided real-valued Wiener process on an appropriate probability space to be specified below. Equation \eqref{1} can model  many processes from chemistry or mathematical biology. For instance, it can be used to describe the evolution of mature populations for age-structured species, where $\Delta u(x,t)$ and $F\left(u(x,t)\right)-\mu u(x,t)$ represent the spatial diffusion and the death rate of mature individuals, $f \left(u(x,t-\tau)\right)-\sigma u(x,t-\tau)$ represents  birth rate, $\sum_{j=1}^{m} g_{j}(x)\frac{\mathrm{d} \omega_{j}(t)}{\mathrm{d}t}$ stands for the random perturbations or environmental effects.

We organize the remaining part of this paper  as follows. In Section 2, we review some results about the existence of random attractors in the existing literature and prove the squeezing property of the RDS generated by \eqref{1} on the random attractors.  Then, we give upper bound of  the Hausdorff  dimension  estimation for \eqref{1}  in Section 3.  At last, we summarize the paper and point out some potential directions for future research in Section 4.

\section{Squeezing property}
This section is concerned about  the squeezing property of the RDS generated by \eqref{1} on the random attractors. In order to adopt the RDS theory, as a first step, we follow the idea of \cite{DLS03} to transform the stochastic \eqref{1} into a random delayed equation, i.e., a path-wise deterministic delayed equation. The same idea has been adopted by many authors when dealing with random attractors or invariant manifolds for various stochastic evolution equations, such as \cite{DLS04,HZ20,HZT,LG20}.

We consider the canonical probability space $(\Omega, \mathcal{F}, P)$  with
$$
\Omega=\left\{\omega=\left(\omega_{1}, \omega_{2}, \ldots, \omega_{m}\right) \in C\left(\mathbb{R} ; \mathbb{R}^{m}\right): \omega_i(0)=0\right\}
$$
and $\mathcal{F}$ is the Borel $\sigma$-algebra induced by the compact open topology of $\Omega,$ while $P$ is the corresponding Wiener measure on $(\Omega, \mathcal{F})$.  Then, we identify $W$ with
$$
W(t,\omega)=\left(\omega_{1}(t), \omega_{2}(t), \ldots, \omega_{m}(t)\right) \quad \text { for } t \in \mathbb{R}.
$$
Moreover, we define the time shift by $$\theta_{t} \omega(\cdot)=\omega(\cdot+t)-\omega(t), t \in \mathbb{R}.$$
Then, $\left(\Omega, \mathcal{F}, P,\left\{\theta_{t}\right\}_{t \in \mathbb{R}}\right)$ is a metric dynamical system.
Consider the stochastic stationary solution of the one dimensional Ornstein-Uhlenbeck equation
\begin{equation}\label{2.1}
\mathrm{d} z_{j}+\mu z_{j} \mathrm{d} t=\mathrm{d} \omega_{j}(t), j=1, \ldots, m,
\end{equation}
which is given by
\begin{equation}\label{2.2}
z_{j}(t) \triangleq z_{j}\left(\theta_{t} \omega_{j}\right)=-\mu \int_{-\infty}^{0} e^{\mu s}\left(\theta_{t} \omega_{j}\right)(s) \mathrm{d} s, \quad t \in \mathbb{R}.
\end{equation}
 Proposition 4.3.3 in \cite{AL} implies that there exists a tempered function $0<r(\omega)<\infty$ such that
\begin{equation}\label{2.3}
\sum_{j=1}^{m}\left|z_{j}\left(\omega_{j}\right)\right|^{2} \leq r(\omega),
\end{equation}
where $r(\omega)$ satisfies, for $P$-a.e. $\omega \in \Omega$,
\begin{equation}\label{2.4}
r\left(\theta_{t} \omega\right) \leq e^{\frac{\mu}{2}|t|} r(\omega), \quad t \in \mathbb{R}.
\end{equation}
Combining \eqref{2.3} with \eqref{2.4}, we obtain that  for $P$-a.e. $\omega \in \Omega$,
\begin{equation}\label{2.5}
\sum_{j=1}^{m}\left|z_{j}\left(\theta_t\omega_{j}\right)\right|^{2} \leq e^{\frac{\mu}{2}|t|} r(\omega), \quad t \in \mathbb{R}.
\end{equation}
Moreover, we have
\begin{equation}\label{2.6}
\sum_{j=1}^{m}\left|z_{j}\left(\theta_{\xi}\omega_{j}\right)\right|^{2} \leq  e^{\frac{\mu\tau}{2}}r(\omega),
\end{equation}
for any $\xi\in [-\tau,0]$ and $P$-a.e. $\omega \in \Omega$.
Putting $z\left(\theta_{t} \omega\right)=\sum_{j=1}^{m} g_{j} z_{j}\left(\theta_{t} \omega_{j}\right)$, we have
$$
\mathrm{d} z+\mu z \mathrm{d} t=\sum_{j=1}^{m} g_{j} \mathrm{d} \omega_{j}.
$$
 Take the transformation $v(t)=u(t)-z\left(\theta_{t} \omega\right)$. Then,  simple computation gives
\begin{equation}\label{2.7}
\displaystyle \frac{dv(x,t)}{dt}=\displaystyle  \Delta v(x,t)-\mu v(x,t)-\sigma v(x,t-\tau)-\sigma z(\theta_{t-\tau}\omega)+F\left(v(t,x)+z(\theta_{t}\omega)\right)+f\left(v(x,t-\tau)+z(\theta_{t-\tau}\omega)\right)+\Delta z(\theta_{t}\omega)
\end{equation}
with boundary condition
\begin{equation}\label{2.8}
v(x,t)=0 \quad \text { for } \quad(x, t) \in \partial \mathcal{O} \times(0, \infty)
\end{equation}
and initial condition
\begin{equation}\label{2.9}
v(x, \xi,\omega)=\psi( x, \xi, \omega) \triangleq \phi(x, \xi)-z\left(\theta_{-\tau} \omega\right) \quad \text { for } \quad(x, \xi) \in \mathcal{O} \times[-\tau, 0].
\end{equation}
 Let  $H=L^2(\mathcal{O})$ be the space of square Lesbegue integrable functions on $\mathcal{O}$ with its usual  norm $\|\cdot\|_{H}$ and inner product $(\cdot,\cdot)_{H}$,  $A=\Delta$ with the domain $D(A)=H^2(D) \cap H_0^1(D)$. Let $X=C([-\tau, 0],H)$ be the space of continuous functions  from $[-\tau, 0]$ to $H$ endowed with the supremum norm $\|\phi\|=\sup_{\theta \in[-r, 0]}\|\phi(\theta)\|_H$ for any $\phi \in X$ and define $u_{t}\in X$  by $u_{t}(\xi)=u(t+\xi)$ for $\xi \in [-\tau, 0]$. Define $L:  X \mapsto H$ and $f : X \mapsto H$ by $L\phi=\sigma\phi(-\tau), f(\phi)=f(\phi(-\tau))$ respectively. Then \eqref{2.7} can be written as the following abstract  SPFDE  in $H$
\begin{equation}\label{2.10}
\displaystyle \frac{dv(t)}{dt}=\displaystyle  Av(t)-\mu v(t)-Lv_t-L z(\theta_{t+\cdot}\omega)+F\left(v(t)+z(\theta_{t}\omega)\right)+f\left(v_t+z(\theta_{t+\cdot}\omega)\right)+Az(\theta_{t}\omega),
\end{equation}
where $\theta_{t+\cdot}\omega$ is defined as $\theta_{t+\xi}\omega$ for $\xi\in [-\tau,0]$.

We follow \cite{DA97},  \cite{LG20} and \cite{25}  to make the following assumptions on $F$ and $f$ throughout the remaining part of this paper.

 Hypothesis A1  $F:  \mathbb{R} \rightarrow \mathbb{R}$ is   a polynomial of odd degree with negative leading coefficient
$$
F(u)=\sum_{k=1}^{2 p-1} a_k u^k, a_{2 p-1}<0.
$$
we assume for simplicity that $p \leq 2$ if $n=3$. $f$ is Lipschitz continuous with $\mathbf{0}$ being a fixed point, that is, $f(\mathbf{0})=\mathbf{0}$ and $$\|f(\phi)-f(\varphi)\|_{H}\leq L_f\|\phi-\varphi\|$$
for any $\phi, \varphi\in X$.

For the pathwise deterministic problem \eqref{2.7}-\eqref{2.9} under Hypothesis A1, it follows from  the Galerkin method as in \cite{C7}  that for $P$-a.e. $\omega \in \Omega$ and $\psi \in X$, it has a unique solution $v(\cdot, \omega, \psi) \in C\left([-\tau, \infty), L^{2}(\mathcal{O})\right) \cap L_{l o c}^{2}\left(\tau, \infty; H_{0}^{1}(\mathcal{O})\right)$. Moreover, this solution is continuous in $\psi \in X$ and $v_{t}(\cdot, \cdot,\psi):(\Omega, \mathcal{F}) \rightarrow\left(X, \mathcal{B}\left(X\right)\right)$ is measurable. By the uniqueness of the solution and semigroup property of $v$,   we can see that the solution $v$ of  \eqref{2.7}-\eqref{2.9} generates a continuous RDS. Let $u(t, \omega, \phi)=v(t, \omega, \psi)+z\left(\theta_{t} \omega\right)$, then $u$ is the solution of \eqref{1} with initial condition $\phi$. We now define a mapping $\Phi: \mathbb{R}^{+} \times \Omega \times X \rightarrow X$ by
$\Phi(t, \omega, \phi)=u_{t}(\cdot, \omega, \phi)=v_{t}(\cdot, \omega,\psi)+z\left(\theta_{t+.} \omega\right)$,
where $u_{t}(\xi, \omega, \phi)=u(t+\xi, \omega, \phi)$ for $\xi \in[-\tau, 0]$. Apparently,  $\Phi$ is an $\mathrm{RDS}$ on $X$ associated with problem \eqref{1}.

It follows from \cite[Theorem 5.1]{25} and \cite[Theorem 4.1]{LG20} that, if Hypothesis A1 holds, then \eqref{1} admits a pullback random attractor $\mathcal{A}(\omega)\subset \mathcal{B}_X(0, c+cr(\omega))$, the ball centered at $0$ with radius  $c+cr(\omega)$ in $X$, where
$c$ is constant defined in \cite[Lemma 3.1]{LG20}. Subsequently, we prove the squeezing property of on the attractor $\mathcal{A}(\omega)$. Unlike the previous works \cite{16,18,LG20,25}, we obtain the results by using the semigroup approach  instead of variation techniques as we work in Banach spaces. It follows from \cite[Theorem 2.6, Chapter 2]{WJ} that  the following  linear equation
\begin{equation}\label{2.11}
\left\{\begin{array}{l}\frac{d\tilde{v}(t)}{dt}=A\tilde{v}(t)-\mu \tilde{v}(t)-L\tilde{v}_t,\\
\tilde{v}_0=\phi \end{array}\right.
\end{equation}
admits a unique global solution $\tilde{v}^\phi(t): [-\tau, \infty)\rightarrow H$.  Define a compact semigroup  $S(t):X\rightarrow X$ by $S(t)\phi=\tilde{v}_t^\phi(\cdot)$ and denote by $A_S$ its infinitesimal generator.

Let $0<\mu_{1} \leq \mu_{2} \leq \cdots \leq \mu_{m} \leq \cdots, \mu_{m} \rightarrow+\infty \quad \text { as } \quad m \rightarrow+\infty$ be eigenvalues of the following eigenvalue problem on $\mathcal{O}$:
 \begin{equation}\label{2.12}
\begin{aligned}
-\Delta u(x)=\mu u(x),\left.\quad u(x)\right|_{x \in \partial \mathcal{O}}=0, \quad x \in \mathcal{O},
\end{aligned}
\end{equation}
with corresponding eigenfunctions $\left\{e_{m}\right\}_{m \in \mathbb{N}}$. Since $A_S$ is compact, it follows from Theorem 1.2 (i) in \cite{WJ} that the spectrum of $A_S$  are point spectra, which we denote by $\varrho_1>\varrho_2>\cdots$ with multiplicity $n_1, n_2,\cdots$. Moreover, it follows from  \cite{WJ} that the characteristic values  $\varrho_1>\varrho_2>\cdots$  of the linear part  $A_S$ are the roots of the following characteristic equation
 \begin{equation}\label{2.13}
\begin{aligned}
\mu^2_{m} -\left(\lambda+\mu-\sigma e^{-\lambda \tau}\right) =0, m=1,2, \cdots.
\end{aligned}
\end{equation}

We  introduce the following state decomposition results of the linear part $A_S$ of \eqref{2.11} established in \cite{WJ}. For any given $\varrho_m<0$, $m\geq 1$, there is a
 \begin{equation}\label{5.8b}
\begin{aligned}
k_m=n_1+n_2+\cdots+n_m
\end{aligned}
\end{equation}
dimensional  subspace $X^U_{k_m}$ such that
$$X=X^U_{k_m} \bigoplus X^S_{k_m}$$
is the decomposition of $X$ by $\varrho_m$. Let $P_{k_m}$ and $Q_{k_m}$ be the projection of $X$ onto $X^U_{k_m}$ and $ X^S_{k_m}$ respectively, that is $X^U_{k_m}=P_{k_m}X$, $X^S_{k_m}=(I-P_{k_m})X=Q_{k_m}X$. It follows from the definition of $P_{k_m}$ and $Q_{k_m}$ that
\begin{equation}\label{5.8a}
\begin{aligned}
\left\|Q_{k_m}S(t) x\right\| & \leq K e^{\varrho_m t}\|x\|, & & t \geq0,
\end{aligned}
\end{equation}
where $K$ is a positive constant. Moreover, there exists a positive constant $M$ such that
 \begin{equation}\label{2.14}
\|S(t)\phi\|\leq Me^{\varrho_1 t}\|\phi\|,
\end{equation}
for any $t\geq0$.

For the later use, we extend the domain of $S(t)$ to the following space of some discontinuous functions
 \begin{equation}\label{2.15}
\begin{aligned}
\hat{C}=\left\{\phi:[-\tau, 0] \rightarrow H;\left.\phi\right|_{[-r, 0)} \quad \text {is continuous and } \lim _{\theta \rightarrow 0^{-}} \phi(\theta) \in H \quad \text{exists} \right\}
\end{aligned}
\end{equation}
and introduce the following  variation of  constants formula established in \cite[Section 4.2, Theorem 2.1]{WJ}
 \begin{equation}\label{2.16}
\begin{aligned}
v_t(\cdot, \omega, \phi) & =S(t) \phi+\int_0^t S(t-s) X_0 [-L z(\theta_{s+\cdot}\omega)+F\left(v(s)+z(\theta_{s}\omega)\right)+f\left(v_s+z(\theta_{s+\cdot}\omega)\right)+Az(\theta_{s}\omega)]d s
\end{aligned}
\end{equation}
 for any $ t \geq 0$, where $X_0:[-\tau, 0] \rightarrow B(X)$ is bounded linear operator on $X$ given by $X_0(\theta)=0$ if $-r \leq \theta<0$ and $X_0(0)=I d$.

\begin{rem}
In general, the solution semigroup defined by \eqref{2.16} have no definition at discontinuous functions and the integral in the formula is undefined as an integral in the phase space. However,  if interpreted correctly, one can see that \eqref{2.16} does make sense. Details can be found in \cite{CM} Pages 144 and 145.
\end{rem}

Subsequently, we prove the squeezing property of the RDS $\Phi$.
\begin{thm}\label{thm2.1}Let $P$ be  the finite dimensional projection $P_{k_m}$ of $X$ onto $X^U_{k_m}$, $K, M, \varrho_m$ and $\rho_1$ being defined by \eqref{5.8a} and \eqref{2.14} respectively and assume Hypothesis A1 holds, then we have
 \begin{equation}\label{2.19}
\begin{aligned}
 \|P[\Phi(t,\omega,\varphi)-\Phi(t,\omega, \psi)]\|
\leq & M e^{(ML_f+\varrho_1)t+ \int_{0}^t  R(\theta_s\omega) d s} \|\varphi-\psi\|
\end{aligned}
\end{equation}
and
 \begin{equation}\label{2.20}
\begin{aligned}
\|(I-P)[\Phi(t,\omega,\varphi)-\Phi(t,\omega, \psi)]\| \leq & [Ke^{ \varrho_m t}+\frac{KM}{\sqrt{2(\varrho_1-\varrho_m)}}e^{[(ML_f+\varrho_1)t+\int_{0}^t MR(\theta_s\omega)ds+\int_{0}^t  R^2(\theta_s\omega)ds]}\\
&+\frac{KML_fe^{[(ML_f+\varrho_1)t+\int_{0}^t MR(\theta_s\omega)ds]}}{\varrho_1-\varrho_m}] \|\varphi-\psi\|
\end{aligned}
\end{equation}
for any $t\geq 0$ and $\varphi, \psi\in\mathcal{A}(\omega)$, where $R(\theta_s\omega)=\sum_{k=1}^{2 p-1} a_k(c+(c+1)r(\theta_s\omega))^{k-1}$ with $r(\omega)$ being defined by \eqref{2.3} and $c$ being  a constant.
\end{thm}
\begin{proof}
For any $\phi, \chi\in \mathcal{A}(\omega)$, denote by $y=\phi-\chi$ and $w_t(\cdot,\omega)=\Phi(t,\omega,\phi)-\Phi(t,\omega, \chi)=u_t(\cdot, \omega, \phi)-u_t(\cdot, \omega, \chi)=v_t(\cdot, \omega, \varphi)-v_t(\cdot, \omega, \psi)$, where $\phi=\varphi+z(\theta_{-\tau}\omega)$ and $\chi=\psi+z(\theta_{-\tau}\omega)$ and hence $y=\phi-\chi$. Then it follows from \eqref{2.16} that
 \begin{equation}\label{2.21}
\begin{aligned}
w_t(\cdot,\omega) = &S(t) y+\int_{0}^t S(t-s) X_0 [F\left(v(s,\omega, \varphi)+z(\theta_{s}\omega)\right)-F\left(v(s,\omega, \psi)+z(\theta_{s}\omega)\right)]d s\\
&+\int_{0}^t S(t-s) X_0 [f\left(v_s(\cdot,\omega, \varphi)+z(\theta_{s+\cdot}\omega)\right)-f\left(v_s(\cdot,\omega, \psi)+z(\theta_{s+\cdot}\omega)\right)]d s.
\end{aligned}
\end{equation}
Taking projection $I-P$ on both sides of \eqref{2.21} leads to
\begin{equation}\label{2.22}
\begin{aligned}
\|(I-P)w_t(\cdot,\omega)\|=&\|(I-P)S(t)y+\int_{0}^t (I-P)S(t-s) X_0 [F\left(v(s,\omega, \varphi)+z(\theta_{s}\omega)\right)-F\left(v(s,\omega, \psi)+z(\theta_{s}\omega)\right)]d s\\
&+\int_{0}^t(I-P)S(t) X_0 [f\left(v_s(\cdot,\omega, \varphi)+z(\theta_{s+\cdot}\omega)\right)-f\left(v_s(\cdot,\omega, \psi)+z(\theta_{s+\cdot}\omega)\right)]d s\|\\
\leq & K e^{\varrho_m t}\|y\| +\|\int_{0}^t (I-P)S(t-s)X_0 [F\left(v(s,\omega, \varphi)\right)-F\left(v(s,\omega, \psi)\right)]d s\|\\
&+L_fM\int_{0}^t  e^{ \varrho_1 t}\|(I-P)w_s\|d s.
\end{aligned}
\end{equation}
Since we have assumed that $F$ is a polynomial with order $2p-1$, then we have
\begin{equation}\label{2.23}
\begin{aligned}
\|F\left(v(t,\omega, \varphi)\right)-F\left(v(t,\omega, \psi)\right)\|_H=&\|\sum_{k=1}^{2 p-1} a_k(v^k(t,\omega, \varphi)-v^k(t,\omega, \psi))\|_H\\
=&\|(v(t,\omega, \varphi)-v(t,\omega, \psi))\sum_{k=1}^{2 p-1} a_k\sum_{j=0}^{k-1}(v^{k-1-j}(t,\omega, \varphi)v^j(t,\omega, \psi))\|_H
\\ \leq & \|(v(t,\omega, \varphi)-v(t,\omega, \psi))\|_H\sum_{k=1}^{2 p-1} a_k\sum_{j=0}^{k-1}\|(v^{k-1-j}(t,\omega, \varphi)\|\|v^j(t,\omega, \psi))\|_H
\\ \leq & \|(v(t,\omega, \varphi)-v(t,\omega, \psi))\|_H\sum_{k=1}^{2 p-1} a_k(c+(c+1)r(\theta_t\omega))^{k-1},
\end{aligned}
\end{equation}
where the last inequality follows from the fact that $v(t,\omega, \varphi)\in \mathcal{A}(\theta_t\omega)-z(\theta_{t}\omega)$, $v(t,\omega, \psi)\in \mathcal{A}(\theta_t\omega)-z(\theta_{t}\omega)$ and $\mathcal{A}(\theta_t\omega)+z(\theta_{t}\omega)\subset \mathcal{B}_X(0, c+(c+1)r(\theta_t\omega))$.
Incorporating \eqref{2.23} into \eqref{2.22} gives
\begin{equation}\label{2.24}
\begin{aligned}
\|(I-P)w_t(\cdot,\omega)\|\leq & K e^{\varrho_m t}\|y\| +M \int_{0}^t R(\theta_s\omega)  e^{\varrho_1 (t-s)}\|(I-P)w_s(\cdot,\omega)\|d s \\
&+L_fM\int_{0}^t  e^{ \varrho_1 (t-s)}\|(I-P)w_s(\cdot,\omega)\|d s,
\end{aligned}
\end{equation}
where $R(\theta_s\omega)=\sum_{k=1}^{2 p-1} a_k(c+(c+1)r(\theta_s\omega))^{k-1}$.
Multiplying both sides of \eqref{2.24} by $e^{- \varrho_1 t}$ gives
\begin{equation}\label{2.25}
\begin{aligned}
e^{- \varrho_1 t}\|(I-P)w_t(\cdot,\omega)\|\leq & K e^{(\varrho_m-\varrho_1) t}\|y\| +M \int_{0}^t R(\theta_s\omega)  e^{-\varrho_1 s}\|(I-P)w_s(\cdot,\omega)\|d s \\
&+L_fM\int_{0}^t  e^{-\varrho_1 s}\|(I-P)w_s(\cdot,\omega)\|d s.
\end{aligned}
\end{equation}
By applying the Gronwall inequality, we have
\begin{equation}\label{2.26}
\begin{aligned}
e^{- \varrho_1 t}\|(I-P)w_t(\cdot,\omega)\|
\leq & K\|y\|[e^{(\varrho_m-\varrho_1)t}+M\int_{0}^t (R(\theta_s\omega)+L_f)  e^{(\varrho_m-\varrho_1) s}e^{\int_{0}^t M(R(\theta_r\omega)+L_f)dr}d s]\\
\leq & K\|y\|\{e^{(\varrho_m-\varrho_1)t}+Me^{ML_ft+M\int_{0}^t R(\theta_s\omega)ds}[\int_{0}^t  R(\theta_s\omega)   e^{(\varrho_m-\varrho_1) s}d s+\frac{L_f}{\varrho_1-\varrho_m}]\}\\
\leq & K\|y\|\{e^{(\varrho_m-\varrho_1)t}+Me^{ML_ft+M\int_{0}^t R(\theta_s\omega)ds}[(\int_{0}^t  R^2(\theta_s\omega)ds)^{\frac{1}{2}}(\int_{0}^t e^{2(\varrho_m-\varrho_1) s}d s)^{\frac{1}{2}}+\frac{L_f}{\varrho_1-\varrho_m}]\}\\
\leq & K\{e^{(\varrho_m-\varrho_1)t}+Me^{ML_ft+M\int_{0}^t R(\theta_s\omega)ds}[e^{\int_{0}^t  R^2(\theta_s\omega)ds}\frac{1}{\sqrt{2(\varrho_1-\varrho_m)}}+\frac{L_f}{\varrho_1-\varrho_m}]\}\|y\|,
\end{aligned}
\end{equation}
where the third inequality follows from the H\"{o}lder's inequality and the last inequality follows from the fact that $\sqrt{x}\leq e^x$.
\eqref{2.26} indicates that
\begin{equation}\label{2.27}
\begin{aligned}
\|(I-P)w_t(\cdot,\omega)\| \leq & [Ke^{ \varrho_m t}+\frac{KM}{\sqrt{2(\varrho_1-\varrho_m)}}e^{[(ML_f+\varrho_1)t+\int_{0}^t MR(\theta_s\omega)ds+\int_{0}^t  R^2(\theta_s\omega)ds]}\\
&+\frac{KML_fe^{[(ML_f+\varrho_1)t+\int_{0}^t MR(\theta_s\omega)ds]}}{\varrho_1-\varrho_m}]\|y\|.
\end{aligned}
\end{equation}

Subsequently, we prove the first part. Since $S(t,\sigma)y=P(t)S(t,\sigma)y+(I-P(t))S(t,\sigma)y$, we have
\begin{equation}\label{2.28}
\begin{aligned}
\|Pw_t(\cdot,\omega)\|=&\|PS(t)y+\int_{0}^t PS(t-s) X_0 [F\left(v(s,\omega, \varphi)+z(\theta_{s}\omega)\right)-F\left(v(s,\omega, \psi)+z(\theta_{s}\omega)\right)]d s\\
&+\int_{0}^tPS(t-s) X_0 [f\left(v_s(\cdot,\omega, \varphi)+z(\theta_{s+\cdot}\omega)\right)-f\left(v_s(\cdot,\omega, \psi)+z(\theta_{s+\cdot}\omega)\right)]d s\|\\
\leq & M e^{\varrho_1 t}\|y\|  +M \int_{0}^t (R(\theta_s\omega)+L_f)  e^{\varrho_1 (t-s)}\| P w_s(\cdot,\omega)\|d s.
\end{aligned}
\end{equation}
Multiplying both sides of \eqref{2.28} by $e^{-\varrho_1 t}$ gives rise to
\begin{equation}\label{2.29}
\begin{aligned}
e^{-\varrho_1 t}\|Pw_t(\cdot,\omega)\|
\leq & M  \|y\|  +M \int_{0}^t (R(\theta_s\omega)+L_f)  e^{-\varrho_1  s }\| P w_s(\cdot,\omega)\|d s.
\end{aligned}
\end{equation}
Again, it follows from the Gronwall inequality that
\begin{equation}\label{2.30}
\begin{aligned}
e^{-\varrho_1 t}\|Pw_t(\cdot,\omega)\|
\leq & M  e^{M \int_{0}^t (R(\theta_s\omega)+L_f)d s}\|y\|,
\end{aligned}
\end{equation}
implying that
\begin{equation}\label{3.30}
\begin{aligned}
 \|Pw_t(\cdot,\omega)\|
\leq & M  e^{(ML_f+\varrho_1)t+ \int_{0}^t  MR(\theta_s\omega) d s}\|y\|.
\end{aligned}
\end{equation}
This proves the first part.
\end{proof}

\section{Hausdorff  dimensions of random attractors}
In this section, we study the Hausdorff  dimension of the attractors $\mathcal{A}(\omega)$ for the RDS generated by \eqref{1}.  The Hausdorff dimension of the random attractor $\mathcal{A}(\omega)\subset X$ is
$$
d_{H}(\mathcal{A}(\omega))=\inf \left\{d: \mu_{H}(\mathcal{A}(\omega), d)= 0 \right\}
$$
where, for $d \geq 0$,
$$\mu_{H}(\mathcal{A}(\omega), d)=\lim _{\varepsilon \rightarrow 0} \mu_{H}(\mathcal{A}(\omega), d, \varepsilon)$$
 denotes the $d$-dimensional Hausdorff measure of the set $\mathcal{A}(\omega)\subset X$, where
 $$\mu_{H}(\mathcal{A}(\omega), d, \varepsilon)=\inf \sum_{i} r_{i}^{d}$$
 and the infimum is taken over all coverings of $\mathcal{A}(\omega)$ by balls of radius $r_{i} \leqslant \varepsilon$. It can be shown that there exists $d_{H}(\mathcal{A}(\omega)) \in[0,+\infty]$ such that $\mu_{H}(\mathcal{A}(\omega), d)=0$ for $d>d_{H}(\mathcal{A}(\omega))$ and $\mu_{H}(\mathcal{A}(\omega), d)=\infty$ for $d<d_{H}(\mathcal{A}(\omega))$. $d_{H}(\mathcal{A})$ is called the Hausdorff dimension of the Hausdorff dimension of $\mathcal{A}(\omega)$.

For a finite dimensional subspace $F$ of a Banach space $X$, denote by $B^F_r(x)$ the ball in $F$ of center $x$ and radius $r$, that is $B^F_r(x)=\{y \in F |\|y-x\|\leq r\}$. It is proved in \cite{30} that the following covering lemma of balls in finite dimensional Banach spaces is true.
\begin{lem}\label{lem3.1}
For every finite dimensional subspace $F$ of a Banach space $X$, we have
 \begin{equation}\label{3.1}
N\left(r_1, B_{r_2}^F\right)\leq m 2^m\left(1+\frac{r_2}{r_1}\right)^m,
\end{equation}
for all $r_1>0, r_2>0$, where $m=\operatorname{dim} F$ and $N\left(r_1, B_{r_2}^F\right)$ is the minimum number of balls needed to cover $B_{r_2}^F$ by the ball of radius $r_1$ calculated in the Banach space $X$.
\end{lem}

We can now establish the upper bound for the Hausdorff dimension of the random attractors for the random dynamical system  $\Phi(t,\omega, \phi)$ in the Banach space $X$ under Hypothesis A1.
\begin{thm}\label{thm3.1} Assume Hypothesis A1 holds and there exists $0<\alpha<2$ such that
 \begin{equation}\label{3.4a}
(\alpha  M  +2 K  +\frac{2KML_f}{\varrho_1-\varrho_m} +\frac{2KM}{\sqrt{2(\varrho_1-\varrho_m)}})e^{(ML_f+\varrho_1+2 \mathbf{E}(R(\theta_t\omega))+2\mathbf{E}(R^2(\theta_t\omega)))t_0}<1.
\end{equation}
Then, the Hausdorff dimension of the global attractor $\mathcal{A}(\omega)$ of \eqref{1} satisfies
 \begin{equation}\label{3.4}
d<\frac{-\ln k_m-k_m\ln(2+\frac{4}{\alpha})}{\ln (\alpha  M  +2 K  +\frac{2KML_f}{\varrho_1-\varrho_m} +\frac{2KM}{\sqrt{2(\varrho_1-\varrho_m)}})+(ML_f+\varrho_1+2 \mathbf{E}(R)+2\mathbf{E}(R^2))t_0},
\end{equation}
where $k_m$ is the dimension of $PX$ defined by \eqref{5.8b}, $K, M, \varrho_m$ and $\rho_1$ being defined by \eqref{5.8a} and \eqref{2.14} respectively and $R(\theta_t\omega)$ is defined in Theorem \ref{thm2.1}.
\end{thm}
\begin{proof}
Since $\mathcal{A}(\omega)$ is a compact subset of $X$, for any $0<\varepsilon<1$, there exist $r_1, \ldots, r_N$ in $(0, \varepsilon]$ and $\tilde{u}_1, \ldots, \tilde{u}_N$ in $X$ such that
 \begin{equation}\label{3.5}
\begin{gathered}
\mathcal{A}(\omega)\subset \bigcup_{i=1}^N B\left(\tilde{u}_i, r_i\right),
\end{gathered}
\end{equation}
where $B(\tilde{u}_i, r_i)$ represents the ball in $X$ of center $\tilde{u}_i$ and radius $r_i$. Without loss of generality, we can assume that for any $i$
 \begin{equation}\label{3.6}
\begin{gathered}
B\left(\tilde{u}_i, r_i\right) \cap \mathcal{A}(\omega) \neq \emptyset,
\end{gathered}
\end{equation}
otherwise, it can be deleted from the sequence $\tilde{u}_1, \ldots, \tilde{u}_N$. Therefore, we can choose  $u_i, i=1,2, \cdots, N$ such that
 \begin{equation}\label{3.7}
\begin{gathered}
u_i \in B\left(\tilde{u}_i, r_i\right) \cap \mathcal{A}(\omega),
\end{gathered}
\end{equation}
and
 \begin{equation}\label{3.8}
\begin{gathered}
\mathcal{A}(\omega) \subset \bigcup_{i=1}^N\left(B\left(u_i, 2 r_i\right) \cap \mathcal{A}(\omega)\right).
\end{gathered}
\end{equation}
It follows from \eqref{2.19} and \eqref{2.20} that for any  $u\in B\left(u_i, 2 r_i\right) \cap \mathcal{A}(\omega)$, we have
 \begin{equation}\label{3.9}
\begin{aligned}
 \|P[\Phi(t_0,\omega,u)-\Phi(t_0,\omega, u_i)]\|
\leq & 2M e^{(ML_f+\varrho_1)t_0+ \int_{0}^{t_0} M R(\theta_s\omega) d s} \|u-u_i\|.
\end{aligned}
\end{equation}
and
 \begin{equation}\label{3.10}
\begin{aligned}
\|(I-P)[[\Phi(t_0,\omega,u)-\Phi(t_0,\omega, u_i)]]\| \leq & 2[Ke^{ \varrho_m t}+\frac{KM}{\sqrt{2(\varrho_1-\varrho_m)}}e^{[(ML_f+\varrho_1)t+\int_{0}^t MR(\theta_s\omega)ds+\int_{0}^t  R^2(\theta_s\omega)ds]}\\
&+\frac{KML_fe^{[(ML_f+\varrho_1)t+\int_{0}^t MR(\theta_s\omega)ds]}}{\varrho_1-\varrho_m}] \|u-u_i\|.
\end{aligned}
\end{equation}

By Lemma \ref{lem3.1}, for any $\alpha>0$, we can find $y_i^1, \ldots, y_i^{n_i}$ such that
 \begin{equation}\label{3.11}
\begin{gathered}
B_{P X}\left(P\Phi(t_0,\omega, u_i), 2 M e^{(ML_f+\varrho_1)t_0+ \int_{0}^{t_0} M R(\theta_s\omega) d s} r_i\right) \subset \bigcup_{j=1}^{n_i} B_{PX}\left(y_i^j, \alpha M e^{(ML_f+\varrho_1)t_0+ \int_{0}^{t_0} M R(\theta_s\omega) d s} r_i\right)
\end{gathered}
\end{equation}
with
 \begin{equation}\label{3.12}
\begin{gathered}
n_i \leq k_m 2^{k_m} \left(1+\frac{2}{\alpha}\right)^{k_m},
\end{gathered}
\end{equation}
where $k_m$ is the dimension of $P X$ and we have denoted by $B_{PX}(y, r)$ the ball in $PX$ of radius $r$ and center $y$.

Set
 \begin{equation}\label{3.13}
\begin{gathered}
u_i^j=y_i^j+(I-P) \Phi(t_0,\omega, u_i)
\end{gathered}
\end{equation}
for $i=1, \ldots, N, j=1, \ldots, n_i$. Then, for any $u\in B\left(u_i, 2 r_i\right) \cap \mathcal{A}(\omega)$,  there exists a $j$ such that
 \begin{equation}\label{3.14}
\begin{aligned}
\left\|\Phi(t_0,\omega,u)-u_i^j\right\|
& \leq\left\|P\Phi(t_0,\omega,u)-y_i^j\right\|+\left\|(I-P) \Phi(t_0,\omega,u)-(I-P) \Phi(t_0,\omega, u_i)\right\| \\
& \leq (\alpha  M e^{(ML_f+\varrho_1)t_0+ \int_{0}^{t_0}M  R(\theta_s\omega) d s}+2(Ke^{ \varrho_m t_0}+\frac{KML_fe^{[(ML_f+\varrho_1)t_0+\int_{0}^t MR(\theta_s\omega)ds]}}{\varrho_1-\varrho_m})\\
&+\frac{2KM}{\sqrt{2(\varrho_1-\varrho_m)}}e^{[(ML_f+\varrho_1)t_0+\int_{0}^{t_0}M R(\theta_s\omega)ds+\int_{0}^{t_0}  R^2(\theta_s\omega)ds]} ) r_i \\
& \leq (\alpha  M  +2 K  +\frac{2KML_f}{\varrho_1-\varrho_m} +\frac{2KM}{\sqrt{2(\varrho_1-\varrho_m)}}) e^{[(ML_f+\varrho_1)t_0+\int_{0}^{t_0}M R(\theta_s\omega)ds+\int_{0}^{t_0}  R^2(\theta_s\omega)ds]} ) r_i
\end{aligned}
\end{equation}
with
 \begin{equation}\label{3.16}
n_i \leq k_m(2+\frac{4}{\alpha} )^{k_m}
\end{equation}
Denote by $\eta=\alpha  M  +2 K  +\frac{2KML_f}{\varrho_1-\varrho_m} +\frac{2KM}{\sqrt{2(\varrho_1-\varrho_m)}}$ and $C(\theta_{t_0}\omega)=e^{[(ML_f+\varrho_1)t_0+\int_{0}^{t_0}M R(\theta_s\omega)ds+\int_{0}^{t_0}  R^2(\theta_s\omega)ds]}$, then we have
 \begin{equation}\label{3.20}
\Phi(t_0,\omega, B\left(u_i, 2 r_i\right)) \cap \mathcal{A}(\theta_{t_0}\omega)  \subset \bigcup_{j=1}^{n_i} B\left(u_i^j, \eta C(\theta_{t_0}\omega) r_i\right).
\end{equation}
Thanks to the invariance of $\mathcal{A}(\omega)$, i.e., $\mathcal{A}(\theta_{t_0}\omega)=\Phi(t_0,\omega, \mathcal{A}(\omega))$, we have
 \begin{equation}\label{3.21}
\mathcal{A}(\theta_{t_0}\omega) \subset \bigcup_{i=1}^N \bigcup_{j=1}^{n_i} B\left(u_i^j, \eta C(\theta_{t_0}\omega) r_i\right) .
\end{equation}
This implies that, for any $d \geq 0$,
 \begin{equation}\label{3.22}
\begin{aligned}
\mu_H\left(\mathcal{A}(\theta_{t_0}\omega), d, \eta C(\theta_{t_0}\omega)\varepsilon\right)
\leq \sum_{i=1}^N \sum_{j=1}^{n_i} \eta^{d}r_i^d \leq k_m(2+\frac{4}{\alpha})^{k_m} \eta^d C^d(\theta_{t_0}\omega)  \sum_{i=1}^N r_i^d,
\end{aligned}
\end{equation}
we deduce, by taking the infimum over all the coverings of $\mathcal{A}(\omega)$ by balls of radii less than $\varepsilon$,
 \begin{equation}\label{3.23}
\begin{aligned}
\mu_H\left(\mathcal{A}(\theta_{t_0}\omega), d, \eta\varepsilon\right)\leq k_m(2+\frac{4}{\alpha})^{k_m} \eta^{d} C^d(\theta_{t_0}\omega)\mu_H(\mathcal{A}(\omega), d, \varepsilon).
\end{aligned}
\end{equation}
Applying the formula recursively for $k$ times and by adopting the fact $\int_{0}^{(k-1)t_0}[R(\theta_r\omega)+ R^2(\theta_s\omega)]dr+\int_{(k-1) t_0}^{k t_0}[R(\theta_r\omega)+ R^2(\theta_s\omega)] ds=\int_{0}^{k t_0}[R(\theta_r\omega)+ R^2(\theta_s\omega)] ds$, we have
 \begin{equation}\label{3.24}
\begin{aligned}
\mu_H\left(\mathcal{A}(\theta_{kt_0}\omega), d, \tilde{C}(\omega)(\eta\varepsilon)^k\right)\leq [k_m(2+\frac{4}{\alpha})^{k_m} \eta^{d}]^k\tilde{C}^d(\omega)\mu_H(\mathcal{A}(\omega), d, \varepsilon),
\end{aligned}
\end{equation}
where $\tilde{C}(\omega)=e^{[(ML_f+\varrho_1)kt_0+\int_{0}^{kt_0}M R(\theta_s\omega)ds+\int_{0}^{kt_0}  R^2(\theta_s\omega)ds]}$.
Thanks to the  ergodicity of $\theta_t$, $t\in \mathbb{R}$, for almost all $\omega\in \Omega$, we have
 \begin{equation}\label{3.25}
\begin{aligned}
\frac{1}{t}[\int_{0}^{t} R(\theta_s\omega)ds+\int_{0}^{t}  R^2(\theta_s\omega)ds]\rightarrow \mathbf{E}(R(\theta_t\omega) )+\mathbf{E}(R^2(\theta_t\omega))
\end{aligned}
\end{equation}
when $t\rightarrow \infty$.
Therefore, there exist $k_0(\omega)$ such that for all $k\geq k_0(\omega)$, we have
\begin{equation}\label{3.26}
\begin{aligned}
 \int_{0}^{kt_0} R(\theta_s\omega)ds+\int_{0}^{kt_0}  R^2(\theta_s\omega)ds]\leq 2 (\mathbf{E}(R(\theta_t\omega) )+\mathbf{E}(R^2(\theta_t\omega)))kt_0.
\end{aligned}
\end{equation}
implying that
\begin{equation}\label{3.27}
\begin{aligned}
\tilde{C}(\omega)\leq e^{(ML_f+\varrho_1+2 \mathbf{E}(R(\theta_t\omega) )+2\mathbf{E}(R^2(\theta_t\omega)))kt_0}
\end{aligned}
\end{equation}
Incorporating \eqref{3.27} into \eqref{3.24}  gives
 \begin{equation}\label{3.28}
\begin{aligned}
\mu_H\left(\mathcal{A}(\theta_{kt_0}\omega), d, \tilde{C}(\omega)(\eta\varepsilon)^k\right)\leq [k_m(2+\frac{4}{\alpha})^{k_m} (\eta e^{(ML_f+\varrho_1+2 \mathbf{E}(R(\theta_t\omega) )+2\mathbf{E}(R^2(\theta_t\omega)))t_0})^{d}]^k \mu_H(\mathcal{A}(\omega), d, \varepsilon),
\end{aligned}
\end{equation}
Hence, if
 \begin{equation}\label{3.25}
d<\frac{-\ln k_m-k_m\ln(2+\frac{4}{\alpha})}{\ln (\alpha  M  +2 K  +\frac{2KML_f}{\varrho_1-\varrho_m} +\frac{2KM}{\sqrt{2(\varrho_1-\varrho_m)}})+(ML_f+\varrho_1+2 \mathbf{E}(R(\theta_t\omega) )+2\mathbf{E}(R^2(\theta_t\omega)))t_0},
\end{equation}
then
 \begin{equation}\label{3.31}
k_m(2+\frac{4}{\alpha})^{k_m}(\eta e^{(ML_f+\varrho_1+2 \mathbf{E}(R(\theta_t\omega) )+2\mathbf{E}(R^2(\theta_t\omega)))t_0})^{d}<1.
\end{equation}
Thus, by taking $k \rightarrow \infty$, we have $(\eta\varepsilon)^k\rightarrow 0$
and \eqref{3.24} leads to
 \begin{equation}\label{3.32}
\mu_H(\mathcal{A}(\theta_{kt_0}\omega), d, (\eta\varepsilon)^k\tilde{C}^d(\omega)) \rightarrow 0.
\end{equation}
By the ergodicity of $\theta_t$, $t\in \mathbb{R}$, we have
 \begin{equation}\label{3.33}
\mu_H(\mathcal{A}(\omega), d,  \varepsilon) \rightarrow 0.
\end{equation}
for almost all $\omega\in \Omega$. This completes the proof.
\end{proof}

\begin{rem}\label{rem2.1}
By \eqref{3.4}, we can see  estimation of $d_H$ depends on the parameter $\alpha$. If we  take $\alpha\uparrow 2$ and assume   $2 M  +2 K  +\frac{2KML_f}{\varrho_1-\varrho_m} +\frac{2KM}{\sqrt{2(\varrho_1-\varrho_m)}})e^{(ML_f+\varrho_1+2 \mathbf{E}(R(\theta_t\omega))+2\mathbf{E}(R^2(\theta_t\omega)))t_0}<1$, then for all $\alpha\in (0,2)$, we have
$\alpha  M  +2 K  +\frac{2KML_f}{\varrho_1-\varrho_m} +\frac{2KM}{\sqrt{2(\varrho_1-\varrho_m)}})e^{(ML_f+\varrho_1+2 \mathbf{E}(R(\theta_t\omega))+2\mathbf{E}(R^2(\theta_t\omega)))t_0}<1$  and hence we obtain the estimation
 \begin{equation}\label{3.4w}
d_{H}\leq \frac{-\ln\Lambda-\Lambda\ln 4}{\ln (2  M  +2 K  +\frac{2KML_f}{\varrho_1-\varrho_m} +\frac{2KM}{\sqrt{2(\varrho_1-\varrho_m)}})+(ML_f+\varrho_1+2 \mathbf{E}(R(\theta_t\omega) )+2\mathbf{E}(R^2(\theta_t\omega)))t_0},
\end{equation}
which is independent of $\alpha$.
\end{rem}

\section{Conclusions}
In this paper, we give an upper bound  of Hausdorff  dimension  of random attractors   for a  stochastic delayed
parabolic equation in Banach spaces that depends only on the inner characteristics of the equation, while not relating to the compact embedding of the phase space to another Banach space as the existing works did. Similar procedures can also be employed to study the fractal dimension. In \cite{KY}, the authors proposed the Lyapunov dimension and gave the famous Kplan-York formula, which was proved to be correct for Navis-Stokes equation  in Hilbert space in \cite{CF}. In \cite{CFT}, Constantin, Foias and Temam also showed the Lyapunov dimension can control the Hausdorff  dimension and gave better estimation of Hausdorff  dimension than the squeeze method, i.e., the method we used here. Furthermore, \cite{F} showed the dimension of chaotic attractor of delayed Mackey-Glass equation satisfies the Kplan-York formula by numerical simulations and Ledrappier  and Young \cite{LY} showed the Kplan-York formula was mathematically correct for random semiflow generated by stochastic ordinary differential equation, which heuristically inspire us to believe the Kplan-York formula also holds for random delay differential equations in Banach spaces. Nevertheless, the rigorous theoretical proof  are  not established even for the deterministic case maybe due to the lack of exterior product and tensor product of the Hilbert space geometry structure, which will be studied in the near future.

\noindent{\bf Acknowledgement.}
This work was jointly supported by the National Natural Science Foundation of China grant 12201379, the Scientific Research Fund of Hunan Provincial Education Department (23C0013), China Scholarship Council(202008430247). \\
The research of T. Caraballo has been partially supported by Spanish Ministerio de Ciencia e
Innovaci\'{o}n (MCI), Agencia Estatal de Investigaci\'{o}n (AEI), Fondo Europeo de
Desarrollo Regional (FEDER) under the project PID2021-122991NB-C21.


\small
\begin{thebibliography}{2}
\bibitem{AL}Arnold L.: Random Dynamical System, Springer-Verlag, New York, Berlin, (1998).
\bibitem{24}Bates P. W., Lu K., Wang B.: Random attractors for stochastic reaction-diffusion equations on unbounded domains. J. Differ. Equ. 246 (2009), 845-869.
\bibitem{17}Bessaih H., Garrido-Atienza M. J., Schmalfu\ss\  B.: Pathwise solutions and attractors for retarded SPDES with time smooth diffusion coefficients. Disc. Contin. Dyn. Syst., 34 (2014), 3945-3968.
\bibitem{C7}Caraballo T.: Nonlinear partial functional differential equations: Existence and stability. J. Math. Anal.
Appl. 262, 87-111 (2001).
\bibitem{CM}Chow S. N., Mallet-Paret J.: Integral averaging and bifurcation. J. Differential Equations, 26 (1977), 112-159.
\bibitem{11}Caraballo T., Langa J. A., Robinson J. C.: Stability and random attractors for a reaction-diffusion equation with multiplicative noise. Discrete Cont. Dyn. Syst., 6 (2000), 875-892.
\bibitem{16}Caraballo T., Garrido-Atienza M. J., Schmalfu\ss\ B.: Existence of exponentially attracting stationary solutions for delay evolution equations. Disc. Contin. Dyn. Syst., 18 (2007), 271-293.
\bibitem{C5}Caraballo T., Sonner S.: Random pullback exponential attractors: General existence results for random dynamical systems in Banach spaces. Disc. Contin. Dyn. Syst., 37 (2017) 6383-6403.
\bibitem{18}Chueshov I., Lasiecka I., Webster J.: Attractors for delayed, non-rotational von Karman plates with applications to ow-structure interactions without any damping. Commun. Partial Differ. Equ., 39 (2014), 1965-1997.
\bibitem{CF}Constantin P.,  Foias, C.: Global Lyapunov exponents, Kaplan-Yorke formulas and the dimension of the attractors for 2D Navier-Stokes equations. Comm. Pure Appl. Math. 38  (1985), 1-27.
    \bibitem{CFT}Constantin P., Foias C., Temam R.: Attractors representing turbulent flows. Memoirs Amer. Math. Soc., 53  (1985), 314.
\bibitem{9}Crauel H., Flandoli F.: Attractors for random dynamical systems. Probab. Theory Relat. Fields, 100 (1994), 365-393.
\bibitem{10}Crauel H.: Random point attractors versus random set attractor. J. London Math. Soc., 63 (2002), 413-427.
\bibitem{CF}Crauel H., Flandoli F.: Hausdorff dimension of invariant sets for random dynamical systems. J. Dynam. Differ. Equ., 10 (1998), 449-474.
\bibitem{CCL}Cui H, Cunha A C, Langa J A.: Finite-dimensionality of tempered random uniform attractors. J. Nonlinear Sci., 32 (2022),  13.
\bibitem{DLS03}Duan J., Lu K., Schmalfu\ss\ B.: Invariant manifolds for stochastic partial differential equations. Ann. Probab., 31 (2003), 2109-2135.
\bibitem{DLS04}Duan J., Lu K., Schmalfu\ss\ B.: Smooth stable and unstable manifolds for stochastic evolutionary equations. J. Dyn. Differ. Equ., 16 (2004), 949-972.
\bibitem{DA97}Debussche A.: On the finite dimensionality of random attractors. Stoch. Anal. Appl., 15 (1997), 473-491.
\bibitem{DA98}Debussche A.: Hausdorff dimension of a random invariant set. J. Math. Pures Appl., 77 (1998), 967-988.
\bibitem{DO}Douady  A., Oesterle J.: Dimension de Hausdorff des attracteurs. C.R. Acad. Paris Ser. A, 290 (1980), 1135-1138.
\bibitem{FX08}Fan X.: Random attractors for damped stochastic wave equations with multiplicative noise. Inter. J. Math., 19 (2008), 421-437.
\bibitem{F}Farmer J. D.: Chaotic attractors of an infinite dimensional dynamical systems. Physica 4D, (1982), 366-393.
 \bibitem{FS}Flandoli F., Schmalfu\ss\ B.: Random attractors for the 3D stochastic navier-stokes equation with multiplicative white noise. Stoch. Stoch.  Proc., 59 (1996), 21-45.
 \bibitem{13}Gao H., Garrido-Atienza M. J., Schmalfu\ss\ B.: Random attractors for stochastic evolution equations driven by fractional Brownian motion. SIAM J. Math. Anal., 46 (2014), 2281-2309.
\bibitem{HZC}Hu W, Caraballo T.: Topological dimensions of random attractors for a stochastic reaction-diffusion equation with delay. https://arxiv.org/abs/2302.05501
\bibitem{HZ20}Hu W., Zhu Q.: Random attractors for a stochastic age-structured population model. J. Math. Phys., 63 (2022), 032703.
\bibitem{HZT}Hu W., Zhu Q., Caraballo T.: Random attractors for a stochastic nonlocal delayed reaction diffusion equation on a semi-infinite interval. IMA J. Appl. Math. doi:10.1093/imamat/hxad025
\bibitem{KY}Kaplan J.,  Yorke, J.:  Chaotic behaviour of multidimensional difference equations. Functional Difference Equations and Approximation of Fixed Points, Lecture Notes in Mathematics 730, Springer-Verlag, Berlin, (1979).
\bibitem{L}Langa J. A.: Finite-dimensional limiting dynamics of random dynamical systems. Dyn. Syst., 18 (2003), 57-68.
\bibitem{LR}Langa J. A., Robinson J. C.: Fractal dimension of a random invariant set. J. Math. Pures Appl., 85 (2006), 269-294.
\bibitem{LY}Ledrappier F., Young, L.S.: Dimension formula for random transformations. Commun. Math. Phys. 117  (1988), 529-548.
\bibitem{LG20}Li S., Guo  S.: Random attractors for stochastic semilinear degenerateparabolic equations with delay. Phys. A, 550 (2020), 124164.
\bibitem{12}Li Y., Guo B.: Random attractors for quasi-continuous random dynamical systems and applications to stochastic reaction-diffusion equations. J. Differ. Equ., 245 (2008), 1775-1800.
\bibitem{30}Man\'{e} R.: On the dimension of the compact invariant sets of certain nonlinear maps, in: Lecture Notes in Math., vol. 898, Springer-Verlag, Berlin/New York, 1981, pp. 230-242.
\bibitem{SB}Schmalfu\ss\ B.: The random attractor of the stochastic Lorenz system. Z. Angew. Math. Phys., 48 (1997), 951-975.
\bibitem{C32}Shirikyan A., Zelik S.: Exponential attractors for random dynamical systems and applications. Stoch. Partial Differ. Equ. Anal. Comput., 1 (2013), 241-281.
\bibitem{SW91}So J.  W-H., Wu J.: Topological dimensions of global attractors for semilinear PDE¡¯s with delays. Bull. Australian Math. Soc., 43 (1991), 407-422.
\bibitem{XHM}Xu L., Huang J, Ma Q.: Random exponential attractor for stochastic non-autonomous suspension bridge equation with additive white noise. Discrete Cont. Dyn. Syst. B, 27 (2022), 6323-6351.
\bibitem{25}Wang X., Lu K., Wang B.: Random attractors for delay parabolic equations with additive noise and deterministic nonautonomous forcing. SIAM J.  Appl.  Dyn.  Syst. 14 (2015), 1018-1047.
\bibitem{15}Wang X., Lu K., Wang B.: Wong-Zakai approximations and attractors for stochastic reaction-diffusion equations on unbounded domains. J. Differ. Equ., 264 (2018), 378-424.
\bibitem{WJ}Wu J.: Theory and applications of partial functional-differential equations, Springer-Verlag, NewYork, 1996.
\bibitem{C37} Zhao M., Zhou S.: Fractal dimension of random invariant sets for nonautonomous random dynamical systems and random attractor for stochastic damped wave equation. Nonlinear Anal., 133 (2016), 292-318.
\bibitem{14}Zhou S.: Random exponential attractors for stochastic reaction-diffusion equation with multiplicative noise in R3. J. Differ. Equ., 263 (2017), 6347-6383.
\textcolor{red}{\bibitem{ZZW}Zhou S., Zhao C.,  Wang Y.: Finite dimensionality and upper semicontinuity of compact kernel sections of non-autonomous lattice systems. Discrete Cont. Dyn. Syst.-A, 21 (2008), 1259-1277.}
\end{thebibliography}
\end{document}